\documentclass[11pt,a4paper]{article}

\usepackage[utf8]{inputenc}
\usepackage[T1]{fontenc}
\usepackage[ngerman,english]{babel}

\usepackage[tbtags]{amsmath}
\usepackage{amssymb}
\usepackage{amsfonts}
\usepackage{amsthm}
\usepackage{bbm}

\usepackage{graphicx}
\usepackage{subcaption}
\usepackage{caption}
\usepackage{epstopdf}

\usepackage[table]{xcolor}
\usepackage{booktabs}
\usepackage{colortbl}
\usepackage{makecell}
\usepackage{array}

\usepackage{tikz}
\usetikzlibrary{shapes,positioning,backgrounds}
\usepackage{pgfplots}
\pgfplotsset{compat=1.18}

\usepackage{a4wide}
\usepackage{algcompatible}
\usepackage{adjustbox}
\usepackage{titlesec}
\usepackage{csquotes}
\usepackage{extarrows}
\usepackage{listings}
\usepackage{ifthen}

\definecolor{headerblue}{RGB}{220,230,242}
\definecolor{rowgray}{RGB}{245,245,245}

\allowdisplaybreaks
\RequirePackage[section,Algorithm]{algorithm}
\RequirePackage{emptypage}

\usepackage[backend=biber,sorting=nyt, giveninits=true]{biblatex}
\renewbibmacro{in:}{}
\DeclareFieldFormat[article]{title}{\textit{#1}} 
\DeclareFieldFormat[article]{journaltitle}{#1}
\theoremstyle{plain}						% Default
\newtheorem{theorem}{Theorem}[section]
\newtheorem{lemma}[theorem]{Lemma}
\newtheorem{proposition}[theorem]{Proposition}
\newtheorem{corollary}[theorem]{Corollary}

\newtheorem{remark}[theorem]{Remark}
\numberwithin{equation}{section}

\titleformat{\section}[block]{\normalfont\bfseries}{\thesection.}{0.5em}{}
\titlespacing{\section}{0pc}{1pc}{1pc}

\titleformat{\subsection}[block]{\normalfont\bfseries}{\thesubsection.}{0.5em}{}
\titlespacing{\subsection}{0pc}{1pc}{1pc}

\begin{document}

	\title{Optimal Inflow Control for Transport Equations with Uncertain Velocities and Demand}
	\author{Simone G\"ottlich\footnotemark[1], \; Thomas Schillinger\footnotemark[1]
    }
	
	\footnotetext[1]{University of Mannheim, Department of Mathematics, 68131 Mannheim, Germany (goettlich@uni-mannheim.de, schillinger@uni-mannheim.de)}

	\date{ \today }
	
	\maketitle
	
	\begin{abstract}
We study optimal inflow control for a linear transport equation subject to uncertainty in both time-dependent downstream demand and transport velocity. The random velocity induces a random travel time and thereby changes the relation between an inflow decision and the demand observed at its arrival time. For a fixed downstream observation window, we derive explicit optimal controls in the continuous-time setting and for piecewise constant controls. We decompose the irreducible stochastic error into contributions from demand and velocity uncertainty and show that the latter admits a bound that is linear in the velocity variance. As a computationally attractive alternative, we analyze a deterministic mean-velocity proxy whose additional performance loss admits a higher-order bound. We also establish Lipschitz stability of the optimal control with respect to perturbations of the velocity law in the Wasserstein distance. Numerical experiments illustrate the analytical results and investigate the proxy strategy for an uncertain nonlocal transport model beyond the linear theory.
\end{abstract}
	
%Optimal control problems involving partial differential equations
%First-order hyperbolic equations
%Partial differential equations with randomness
	
	
	%%%%%%%%%%%%%%%%%%%%%%%%%%%%%%%%%%%%%%%%%%%%

	\section{Introduction}

Transport equations provide a basic modeling framework for systems in which goods, vehicles, energy, or information are propagated through space. Applications range from production and conveyor systems \cite{Armbruster2006} and traffic flow \cite{GARAVELLO.2016,GARAVELLO.2006} to gas transport \cite{Banda.2006}, electricity transmission \cite{Goettlich2016}, and models arising in the life sciences \cite{Formaggia2003}. In many of these applications, control decisions have to be made in the presence of uncertainty. In particular, two sources of uncertainty are of direct relevance for transport processes: the downstream demand may fluctuate in time, while the transport dynamics themselves may be subject to random variations. The latter affect the travel time through the system and therefore directly influence when an inflow decision becomes visible at the downstream boundary.

Optimal control of hyperbolic partial differential equations has been studied extensively; we refer to \cite{Herty2023} for an overview. Existence, controllability, and control strategies are available for broad classes of conservation laws and transport systems \cite{Ancona1998,Bressan2002}, with applications including traffic flow management \cite{Bayen2022} and gas transport networks \cite{Hintermueller2024,Gugat2011}. At the same time, uncertainty has increasingly been incorporated into hyperbolic models. It may enter through uncertain initial data, for instance as a consequence of incomplete measurements \cite{Iacomini.2022}, through the dynamics, such as fluctuating velocities, machine failures, or traffic incidents \cite{Barth2017,Brencher.2020,DEGOND.2007,Goettlich.2011,Gottlich.2021}, or through boundary data, in particular uncertain consumer demand \cite{Gugat2018,Zixuan2025}.

Optimal control and optimization problems governed by PDEs with uncertain coefficients form a broader related research area. Existing contributions address, among other topics, pathwise optimal controls and multilevel Monte Carlo approximation for PDEs with random coefficients \cite{AliUllmannHinze2017}, sparse optimal control under uncertain parameters \cite{LiStadler2019}, and performance bounds for approximation-based PDE-constrained optimization under uncertainty \cite{ChenRoyset2023}. A different but related issue arises when the probability law of the uncertainty itself is not known exactly. This has motivated data-driven and distributionally robust optimization based on Wasserstein ambiguity sets \cite{GaoKleywegt2023,EsfahaniKuhn2018} and, in stochastic control, Wasserstein-robust dynamic and linear-quadratic control formulations \cite{KimYang2023,Yang2021}.
These approaches typically modify the optimization problem by hedging against a family of probability distributions. Our perspective is different: we keep the prescribed stochastic control problem unchanged and quantify directly the sensitivity of its optimizer with respect to perturbations of the velocity law. This leads to Lipschitz stability of the optimal inflow in the Wasserstein--1 distance and to a quadratic performance bound under distributional misspecification.

For control problems, these different sources of uncertainty have largely been considered separately. In particular, optimal inflow control under uncertain demand has been investigated for deterministic transport dynamics \cite{Goettlich2022,Gugat2018,Zixuan2025}, whereas stochastic transport models with uncertain velocities have mainly been studied from the perspective of well-posedness and uncertainty propagation \cite{ABGRALL.2017,Brencher.2020,Dorini2020,Dorini2011}. The simultaneous presence of uncertain demand and uncertain transport dynamics leads to an additional difficulty: a random transport velocity induces a random travel time. Consequently, an inflow prescribed at a fixed time contributes to the outflow at a random future time. The optimal control therefore cannot, in general, be obtained by simply shifting the expected demand by a deterministic transport delay.

In this work, we investigate this interaction for a linear transport equation with a positive random, time-independent velocity and stochastic time-dependent demand. The linear transport equation is sufficiently simple to permit an explicit analysis, while retaining the essential mechanism generated by uncertain travel times. Moreover, it constitutes a fundamental building block for more complex hyperbolic models and appears, for example, in production and supply systems, simplified electricity transmission models \cite{Goettlich2016}, and recent models of tram networks \cite{Schillinger2026}. Our analysis builds on the stochastic solution framework of \cite{ABGRALL.2017,Brencher.2020} and extends the deterministic-velocity control setting of \cite{Goettlich2022}.

A central modeling choice is to evaluate the mismatch between outflow and demand on a fixed downstream observation window. This is natural when demand is prescribed with respect to physical time, but it has an important consequence for random travel times: near the beginning and end of the control horizon, only a subset of the possible velocity realizations contributes to the objective. The optimal inflow is therefore determined by a conditional average over admissible arrival times. Away from these temporal boundaries, all velocity realizations contribute, and the optimal control is obtained by averaging the expected future demand over the full velocity distribution. This distinction between temporal boundary effects and the genuine random-travel-time mechanism is central to the analysis below.

A second question is whether the additional uncertainty in the transport velocity has to be incorporated explicitly into the control design. From a computational perspective, a natural alternative is to replace the random velocity by its mean and solve the resulting deterministic proxy problem. The explicit structure of the linear model allows us not only to compare these strategies, but also to distinguish the error caused by the uncertain travel time itself from the additional approximation error introduced by the proxy.

The main contributions of this work are the derivation of explicit optimal inflow controls under simultaneous uncertainty in demand and transport velocity, both in continuous time and for piecewise constant controls. We establish convergence of the discrete controls and quantify the associated temporal discretization error. Moreover, we separate the effects of demand and random travel-time uncertainty and analyze the additional error introduced by a deterministic mean-velocity proxy. Finally, we prove stability of the optimal control with respect to perturbations of the velocity distribution in the Wasserstein distance.

The remainder of the paper is organized as follows. Section~\ref{sec:standardFramework} recalls the deterministic-velocity reference problem with uncertain demand and fixes the notation. Section~\ref{sec:speedRV} introduces the random transport velocity, derives the continuous and piecewise constant optimal controls, and discusses the deterministic mean-velocity proxy. Section~\ref{sec:errorAnalysis} separates the different error mechanisms and establishes the corresponding asymptotic estimates. Section~\ref{sec:summaryControls} summarizes the analytical results. Section~\ref{sec:numericalAnalysis} presents the numerical experiments, including the extension to the nonlocal transport model. Finally, Section~\ref{sec:conclusion} summarizes the main findings and outlines directions for further research.

\section{Deterministic transport with uncertain demand}
\label{sec:standardFramework}

Before introducing uncertainty in the transport dynamics, we briefly recall the deterministic-velocity setting with stochastic demand discussed in \cite{Goettlich2022}. This provides the reference problem for the stochastic model developed in Section~\ref{sec:speedRV} and fixes the notation used throughout the paper.
We consider a supply line on the spatial domain $[0,1]$. Material enters at $x=0$ through a time-dependent inflow $u(t)$ and is transported with a constant velocity
\[
0<\lambda_{\min}\leq \lambda\leq\lambda_{\max}.
\]
Denoting the material density by $\rho(x,t)$, the dynamics are governed by
\begin{align}
    \rho_t(x,t)+\lambda\rho_x(x,t)&=0,
    & \rho(x,0)&=\rho_0(x),
    & \lambda\rho(0,t)&=u(t),
    \qquad x\in(0,1),\ t\in[0,T].
    \label{eq:PDE}
\end{align}
By the method of characteristics \cite{LEVEQUE.2002},
\begin{align*}
\rho(x,t)=
\begin{cases}
\rho_0(x-\lambda t), & x-\lambda t\geq 0,\\[1mm]
\dfrac{1}{\lambda}u\!\left(t-\dfrac{x}{\lambda}\right), & x-\lambda t<0.
\end{cases}
\end{align*}
Consequently, the outflow at $x=1$ is
\begin{align}
f^{\mathrm{out}}(t)=\lambda\rho(1,t)=
\begin{cases}
\lambda\rho_0(1-\lambda t), & \lambda t\leq 1,\\[1mm]
u\!\left(t-\dfrac{1}{\lambda}\right), & \lambda t>1.
\end{cases}
\label{eq:fout}
\end{align}
Thus, once the influence of the initial condition has left the domain, the inflow reaches the downstream boundary after the deterministic transport delay $1/\lambda$. This simple time-shift relation is the key structural property that will change when $\lambda$ becomes random.

\paragraph{Stochastic demand.}
At the downstream boundary, the outflow is compared with a time-dependent, uncertain consumer demand $(D(t))_{t\in[0,T]}$, modeled as a stochastic process on a probability space $(\Omega,\mathcal{F},\mathbb{P})$. We assume that $D(0)=d_0\geq0$ is known, that $D(t)\geq0$ almost surely, and that
\[
D\in L^2(\Omega\times(0,T)).
\]
In particular, the first and second moments of the demand are well defined. The analysis below only relies on these moments and is therefore independent of a particular stochastic model.
Typical examples satisfying these assumptions include the Cox--Ingersoll--Ross process, the Jacobi process, and the exponential of an Ornstein--Uhlenbeck process; see \cite{Iacus2008}. Throughout the analytical part of this paper, however, we do not impose a specific stochastic model for the demand.

\paragraph{Optimal inflow.}
For deterministic demand, \eqref{eq:fout} would allow exact demand tracking by shifting the demand by the transport time $1/\lambda$. Under stochastic demand, exact pathwise tracking is in general impossible. We therefore minimize the expected squared mismatch,
\begin{align*}
\min_{u\in L^2(0,T-1/\lambda)}\quad
H(u)
:=
\int_{\frac{1}{\lambda}}^{T}
\mathbb{E}\!\left[
\left(f^{\mathrm{out}}(s)-D(s)\right)^2
\right]\,\mathrm{d}s,
\qquad
\text{subject to \eqref{eq:PDE}--\eqref{eq:fout}.}
%\label{eq:ObjFct}
\end{align*}
Since $f^{\mathrm{out}}(s)$ is deterministic in the present setting,
\[
\mathbb{E}\!\left[
\left(f^{\mathrm{out}}(s)-D(s)\right)^2
\right]
=
\mathbb{E}[D(s)^2]
-2\,\mathbb{E}[D(s)]f^{\mathrm{out}}(s)
+\left(f^{\mathrm{out}}(s)\right)^2.
\]
Pointwise minimization with respect to the outflow therefore gives
$f^{\mathrm{out}}(s)=\mathbb{E}[D(s)]$. Using \eqref{eq:fout}, the optimal continuous-time inflow is
\begin{align}
u_\text{det}^*(t)
=
\mathbb{E}\!\left[
D\!\left(t+\frac{1}{\lambda}\right)
\right],
\qquad
t\in\left[0,T-\frac{1}{\lambda}\right].
\label{eq:optInflowEasy}
\end{align}
Hence, for deterministic transport dynamics, uncertainty in the demand affects the optimal control only through its mean; the transport mechanism itself contributes the fixed time shift $1/\lambda$. We refer to \cite{Goettlich2022} for a more detailed analysis, including time-dependent velocities and tree-shaped networks.

In applications, the inflow may only be updated at prescribed time instants. Let
\begin{align*}
\Pi_{n,\text{det}}
=
\left\{
0=t_0^{(n)}<t_1^{(n)}<\dots<t_n^{(n)}
<t_{n+1}^{(n)}
=
T-\frac{1}{\lambda}
\right\}
\end{align*}
be a partition and let $U_{n,\text{det}}\subset L^2(0,T-1/\lambda)$ denote the corresponding space of piecewise constant controls. Suppressing the superscript on the grid points, we write
\begin{align*}
u_\text{det}^{(n)}(t)=u_{i,\text{det}}^{(n)}
\qquad\text{for }t\in[t_i,t_{i+1}),
\quad i=0,\ldots,n,
%\label{eq:pwconstantU}
\end{align*}
with the natural convention at the right endpoint. Restricting $H$ to $U_n$ and differentiating with respect to each $u_i^{(n)}$ yields
\begin{align*}
u_{i,\text{det}}^{(n),*}
=
\frac{1}{t_{i+1}-t_i}
\int_{t_i}^{t_{i+1}}
\mathbb{E}\!\left[
D\!\left(s+\frac{1}{\lambda}\right)
\right]\,\mathrm{d}s,
\qquad i=0,\ldots,n.
%\label{eq:optimalU_iDeterministic}
\end{align*}
Thus, the optimal piecewise constant control is the cell average of the optimal continuous-time control \eqref{eq:optInflowEasy}. This deterministic reference case will be useful below: once the transport velocity becomes random, the fixed delay $1/\lambda$ is replaced by a distribution of possible arrival times, and the structure of the optimal piecewise constant control changes accordingly.

\section{Random transport velocity and optimal inflow control}
\label{sec:speedRV}

We now introduce uncertainty in the transport dynamics. In contrast to Section~\ref{sec:standardFramework}, the transport delay is no longer fixed but depends on the realization of a random velocity, as depicted in Figure \ref{fig:uncertainDelay}. We retain a fixed downstream observation window, which is natural when the demand is prescribed with respect to physical time. This leads to temporal boundary effects in the optimal control that have to be treated explicitly.
\begin{figure}
\centering
    \begin{tikzpicture}[scale=1, every node/.style={font=\small}]
			% --- deterministic row ---
			\draw[->, thick] (0,2.1) -- (9,2.1);
			\filldraw (0,2.1) circle (2pt);
			\node[above left] at (0,2.15) {$t$};
			\draw[->, ultra thick, blue!70!black] (0,2.1) to[bend left=18] (4.5,2.1);
			\filldraw (4.5,2.1) circle (2pt);
			\node[above] at (4.5,2.4) {$t+1/\lambda$};
			\node[left] at (-0.4,2.1) {\textbf{deterministic} $\lambda$};
			
			% --- random row ---
			\draw[->, thick] (0,0) -- (9,0);
			\filldraw (0,0) circle (2pt);
			\node[below left] at (0,-0.05) {$t$};
            \node[below] at (3.2,-0.05) {$t+1/\lambda_{\max}$};
            \draw (3.2,0.1) -- (3.2,-0.1);
            \filldraw (4.5,2.1) circle (2pt);
            \node[below] at (6.7,-0.05) {$t+1/\lambda_{\min}$};
            \draw (6.7,0.1) -- (6.7,-0.1);
			\foreach \x/\op in {3.2/0.15, 3.7/0.35, 4.2/0.6, 4.7/0.85, 5.2/1, 5.7/0.7, 6.2/0.4, 6.7/0.2} {
				\draw[->, thick, red!70!black, opacity=\op] (0,0) to[bend left=10] (\x,0);
			}
			% little density bump above the arrival cloud
		%	\draw[smooth, gray!60!black, thick] plot coordinates {(3.6,0.55) (4.4,1.1) (5.2,1.35) (6.0,1.05) (6.8,0.5)};
			\node[left] at (-0.4,0) {\textbf{random} $\lambda$};
			%\node[below] at (5.0,-0.5) {spread of arrival times};
		\end{tikzpicture}
         \caption{The influence of the uncertain transport velocity on the transport delay.}
    \label{fig:uncertainDelay}
\end{figure}
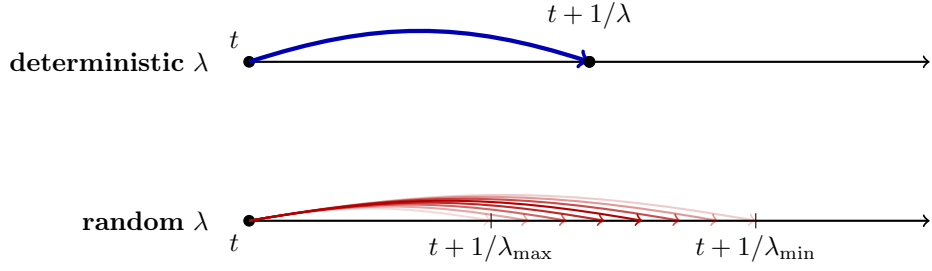

\subsection{Stochastic transport model and random travel times}
% Let $(\Omega,\mathcal A,\mathbb P)$ be a probability space and let
% \[
% \lambda:\Omega\to[\lambda_{\min},\lambda_{\max}],
% \qquad 0<\lambda_{\min}<\lambda_{\max},
% \]
% be a positive random variable. 
Let
\[
\lambda:\Omega\to[\lambda_{\min},\lambda_{\max}],
\qquad 0<\lambda_{\min}<\lambda_{\max},
\]
be a positive random variable on the probability space $(\Omega,\mathcal F,\mathbb P)$ introduced above.
For each realization $\omega\in\Omega$, we consider
\begin{align}
\rho_t(x,t;\omega)+\lambda(\omega)\rho_x(x,t;\omega)&=0,
&
\rho(x,0;\omega)&=\rho_0(x),
&
\lambda(\omega)\rho(0,t;\omega)&=u(t).
\label{eq:stochasticPDE}
\end{align}
By characteristics, the downstream outflow is
\begin{align*}
f^{\mathrm{out}}(t,\omega)
=
\begin{cases}
\lambda(\omega)\rho_0(1-\lambda(\omega)t),
& \lambda(\omega)t\leq1,\\[1mm]
u\!\left(t-\dfrac{1}{\lambda(\omega)}\right),
& \lambda(\omega)t>1.
\end{cases}
\end{align*}

Throughout the analytical part, we assume:
\begin{enumerate}
\item[(A1)] $\lambda_{\min}\leq\lambda\leq\lambda_{\max}$ almost surely, with $\lambda_{\min}>0$.
\item[(A2)] $F_\lambda\in C^1([\lambda_{\min},\lambda_{\max}])$ is the cumulative distribution function of $\lambda$ with density $\phi_\lambda$.
\item[(A3)] $D:[0,T]\times\Omega\to[0,\infty)$ is jointly measurable and $D\in L^2(\Omega\times(0,T))$.
\item[(A4)] $m(t):=\mathbb E[D(t)]$ denotes the mean demand.
\item[(A5)] The random variable $\lambda$ is independent of the demand process $D$, i.e.,
$\sigma(\lambda)$ and $\sigma(D(t):t\in[0,T])$ are independent.
\end{enumerate}
The random travel time through the unit interval is
\[
\tau:=\frac{1}{\lambda}
\in\left[\frac{1}{\lambda_{\max}},\frac{1}{\lambda_{\min}}\right].
\]
We compare outflow and demand on the fixed downstream observation window
\[
I_{\mathrm{obs}}
:=
\left[\frac{1}{\lambda_{\min}},T\right].
\]
Since an outflow observed at time $s\in I_{\mathrm{obs}}$ may originate from an inflow as late as $s-1/\lambda_{\max}$, the admissible control has to be defined on
\[
I_{\mathrm{ctrl}}
:=
\left[0,T-\frac{1}{\lambda_{\max}}\right].
\]
The stochastic optimal control problem is therefore
\begin{align*}
\min_{u\in L^2(I_{\mathrm{ctrl}})}
H(u)
:=
\int_{1/\lambda_{\min}}^T
\mathbb E\!\left[
\left(
u\!\left(s-\frac{1}{\lambda}\right)-D(s)
\right)^2
\right]\,\mathrm ds .
\end{align*}
Here the contribution of the initial condition has disappeared because $s\geq1/\lambda_{\min}$ implies $\lambda s\geq1$ almost surely.

For a fixed control time $t\in I_{\mathrm{ctrl}}$, only those realizations of $\lambda$ for which the corresponding arrival time belongs to $I_{\mathrm{obs}}$ contribute to the objective. We therefore introduce
\begin{align}
\Lambda(t)
:=
\left\{
\ell\in[\lambda_{\min},\lambda_{\max}]
:
\frac{1}{\lambda_{\min}}
\leq t+\frac{1}{\ell}\leq T
\right\},
\qquad
q(t):=\int_{\Lambda(t)}\phi_\lambda(\ell)\,\mathrm d\ell.
\label{eq:admissibleVelocitySet}
\end{align}
If $q(t)=0$, the value of the control at $t$ does not affect the objective; all formulas below are therefore understood on the set where $q(t)>0$.

\subsection{Continuous-time optimal control}
\label{sec:timeContinuous}

Changing variables pathwise according to $t=s-1/\lambda$ and using (A5), the terms in $H$ that depend on $u$ can be written as
\[
\int_{I_{\mathrm{ctrl}}}
\int_{\Lambda(t)}
\left[
u(t)^2
-
2u(t)m\!\left(t+\frac{1}{\ell}\right)
\right]
\phi_\lambda(\ell)\,\mathrm d\ell\,\mathrm dt.
\]
This representation makes the pointwise minimization with respect to $u(t)$ explicit.

\begin{theorem}[Optimal continuous-time inflow]
\label{thm:continuousOptimalControl}
Let \emph{(A1)--(A5)} hold. Then, for almost every $t\in I_{\mathrm{ctrl}}$ with $q(t)>0$, the unique optimal control is
\begin{align}
u^*(t)
=
\frac{
\displaystyle
\int_{\Lambda(t)}
m\!\left(t+\frac{1}{\ell}\right)\phi_\lambda(\ell)\,\mathrm d\ell
}{
q(t)
%\displaystyle \int_{\Lambda(t)}\phi_\lambda(\ell)\,\mathrm d\ell
}.
\label{eq:continuousOptControlGeneral}
\end{align}

Equivalently,
\[
u^*(t)
=
\mathbb E\!\left[
D\!\left(t+\frac{1}{\lambda}\right)
\,\middle|\,
t+\frac{1}{\lambda}\in I_{\mathrm{obs}}
\right].
\]
\end{theorem}

\begin{proof}
For fixed $t$, the $u$-dependent part of the integrand is the strictly convex quadratic function
\[
q(t)u(t)^2
-
2u(t)
\int_{\Lambda(t)}
m\!\left(t+\frac{1}{\ell}\right)\phi_\lambda(\ell)\,\mathrm d\ell.
\]
Its unique minimizer is \eqref{eq:continuousOptControlGeneral}. The conditional-expectation representation follows from (A5).
\end{proof}

The formula separates the genuine random-delay effect from the temporal boundary effect. In the interior control interval
\begin{align}
I_{\mathrm{int}}
:=
\left[
\frac{1}{\lambda_{\min}}-\frac{1}{\lambda_{\max}},
\,
T-\frac{1}{\lambda_{\min}}
\right],
\label{eq:interiorControlInterval}
\end{align}
provided this interval is nonempty, every velocity realization produces an arrival time inside $I_{\mathrm{obs}}$. Hence $\Lambda(t)=[\lambda_{\min},\lambda_{\max}]$ and $q(t)=1$.

\begin{corollary}[Interior optimal control]
\label{cor:interiorOptimalControl}
For $t\in I_{\mathrm{int}}$, the optimal control reduces to
\begin{align}
u^*(t)
=
\mathbb E\!\left[
D\!\left(t+\frac{1}{\lambda}\right)
\right]
=
\int_{\lambda_{\min}}^{\lambda_{\max}}
m\!\left(t+\frac{1}{\ell}\right)\phi_\lambda(\ell)\,\mathrm d\ell.
\label{eq:continuousOptCOntrol}
\end{align}
\end{corollary}

Thus, away from the temporal boundaries, the deterministic time shift is replaced by an average over the complete distribution of random travel times. Near the beginning and end of the control horizon, the average is taken only over those velocities whose associated arrival times lie in the fixed observation window.

\subsection{Piecewise constant controls}
\label{sec:pwConstant}

Let 
\[
\Pi_n
=
\left\{
0=t_0^{(n)}<\cdots<t_{n+1}^{(n)}
=
T-\frac{1}{\lambda_{\max}}
\right\}
\]
be a partition of $I_{\mathrm{ctrl}}$, and let $U_n$ denote the corresponding space of piecewise constant controls. We suppress the superscript on the grid points and write
\[
u^{(n)}(t)=u_i^{(n)}
\qquad\text{for }t\in[t_i,t_{i+1}).
\]
Define
\begin{align*}
Q_i^{(n)}
:=
\int_{t_i}^{t_{i+1}}q(t)\,\mathrm dt,
%\\
\quad \text{and} \quad
M_i^{(n)}
:=
\int_{t_i}^{t_{i+1}}
\int_{\Lambda(t)}
m\!\left(t+\frac{1}{\ell}\right)
\phi_\lambda(\ell)\,\mathrm d\ell\,\mathrm dt .
\end{align*}

\begin{proposition}[Optimal piecewise constant inflow]
\label{prop:optimalDiscreteControl}
Under \emph{(A1)--(A5)}, for every cell with $Q_i^{(n)}>0$, the unique optimal coefficient is
\begin{align}
u_i^{(n),*}
=
\frac{M_i^{(n)}}{Q_i^{(n)}}
=
\frac{
\displaystyle
\int_{t_i}^{t_{i+1}}q(t)u^*(t)\,\mathrm dt
}{
\displaystyle
\int_{t_i}^{t_{i+1}}q(t)\,\mathrm dt
}.
\label{eq:optimalDiscreteUi}
\end{align}
\end{proposition}

\begin{proof}
The $u_i^{(n)}$-dependent contribution of the cell $[t_i,t_{i+1})$ is
\[
Q_i^{(n)}(u_i^{(n)})^2-2M_i^{(n)}u_i^{(n)}.
\]
Strict convexity and the first-order optimality condition yield \eqref{eq:optimalDiscreteUi}. Multiplying \eqref{eq:continuousOptControlGeneral} by $q(t)$ yields the second equality.
\end{proof}

Hence, the discrete optimum is a $q$-weighted cell average of the continuous optimum. On every cell contained in $I_{\mathrm{int}}$, where $q(t)=1$, this reduces to the ordinary cell average
\begin{align}
\label{eq:pwConstantInterior}
    u_i^{(n),*}
=
\frac{1}{t_{i+1}-t_i}
\int_{t_i}^{t_{i+1}}u^*(t)\,\mathrm dt.
\end{align}
The relation between the optimal piecewise constant and continuous-time controls is established in Lemma \ref{lem:discretizationError}.

\subsection{Deterministic mean-velocity proxy}
\label{sec:proxyModel}
We finally replace the random velocity by its mean
\[
\bar\lambda:=\mathbb E[\lambda].
\]
In the stochastic setting, the calculation of the optimal control requires the probability density of the random velocity to be known and involves an additional integration to compute its expected value. We therefore consider this mean-velocity proxy as an approximation that enables a faster computation of the optimal control.
To evaluate the proxy within the same fixed observation window, define
\[
I_{\mathrm{proxy}}
:=
\left[
\frac{1}{\lambda_{\min}}-\frac{1}{\bar\lambda},
\,
T-\frac{1}{\bar\lambda}
\right].
\]
On this interval, the deterministic proxy control obtained from Section~\ref{sec:standardFramework} is
\begin{align}
\bar u(t)
=
%m\!\left(t+\frac{1}{\mathbb E[\lambda]}\right).
m\!\left(t+\frac{1}{\color{teal}\bar\lambda}\right).
\label{eq:ProxyControlContinuous}
\end{align}
For a piecewise constant implementation on a cell $[t_i,t_{i+1})\subset I_{\mathrm{proxy}}$, the corresponding proxy coefficient is
\begin{align}
\bar u_i^{(n)}
=
\frac{1}{t_{i+1}-t_i}
\int_{t_i}^{t_{i+1}}
%m\!\left(t+\frac{1}{\mathbb E[\lambda]}\right)\,\mathrm dt.
m\!\left(t+\frac{1}{\color{teal}\bar{\lambda}}\right)\,\mathrm dt.
\label{eq:ProxyControlPWConstant}
\end{align}
% On the common interior region where both formulas are unaffected by temporal boundary truncation, the stochastic optimum and the proxy take the particularly transparent forms derived in 
% \[
% u^*(t)
% =
% \mathbb E_\lambda\!\left[
% m\!\left(t+\frac{1}{\lambda}\right)
% \right],
% \qquad
% \bar u(t)
% =
% %m\!\left(t+\frac{1}{\mathbb E[\lambda]}\right).
% m\!\left(t+\frac{1}{\color{teal}\bar{\lambda}}\right).
% \]
The proxy therefore replaces averaging after the nonlinear travel-time transformation $\lambda\mapsto1/\lambda$ by evaluation at the mean velocity. Section~\ref{sec:errorAnalysis} investigates the resulting performance loss.

\section{Error analysis}
\label{sec:errorAnalysis}

The representation of the optimal control in Section~\ref{sec:speedRV} suggests a natural separation of the different error mechanisms. The fixed observation window produces additional temporal boundary effects, whereas the genuine influence of the random transport velocity is most transparent on the interior interval $I_{\mathrm{int}}$ from \eqref{eq:interiorControlInterval}. We therefore first isolate the boundary contribution and then analyze, on $I_{\mathrm{int}}$, (i) the irreducible stochastic error, (ii) the additional error caused by velocity uncertainty, (iii) the deterministic mean-velocity proxy, and (iv) temporal control discretization.

\subsection{A variance--bias decomposition}
For $t\in I_{\mathrm{ctrl}}$, define $q(t)$ as in Equation \eqref{eq:admissibleVelocitySet}.
% \[
% q(t)
% =
% \int_{\Lambda(t)}\phi_\lambda(\ell)\,\mathrm d\ell,
% \qquad
% g(t)
% =
% \frac{1}{q(t)}
% \int_{\Lambda(t)}
% m\!\left(t+\frac{1}{\ell}\right)
% \phi_\lambda(\ell)\,\mathrm d\ell
% \]
%whenever $q(t)>0$. By Theorem~\ref{thm:continuousOptimalControl}, $g=u^*$ almost everywhere.
Completing the square in the representation of $H$ yields 
\begin{align}
H(u) &= \int_{I_{\mathrm{ctrl}}}\mathbb{E}\left[\mathbbm{1}_{{\{t+1/\lambda\in I_{\mathrm{obs}}\}}} \left(u(t)  - D\left(t+\frac{1}{\lambda}\right)\right)^2 \right] \mathrm dt \nonumber\\
&= \int_{I_{\mathrm{ctrl}}} q(t)\,u(t)^2
-2u(t)\,\mathbb E\!\left[\mathbf 1_{\{t+1/\lambda\in I_{\mathrm{obs}}\}}D\!\left(t+\tfrac1\lambda\right)\right] +\mathbb E\!\left[\mathbf 1_{\{t+1/\lambda\in I_{\mathrm{obs}}\}}D\!\left(t+\tfrac1\lambda\right)^2\right]   \mathrm dt\nonumber\\ 
&= \int_{I_{\mathrm{ctrl}}} q(t)\,u(t)^2
-2u(t)q(t)u^*(t) +\mathbb E\!\left[\mathbf 1_{\{t+1/\lambda\in I_{\mathrm{obs}}\}}D\!\left(t+\tfrac1\lambda\right)^2\right]   \mathrm dt\nonumber\\ 
&= \int_{I_{\mathrm{ctrl}}} q(t)\,u(t)^2
-2u(t)q(t)u^*(t) +q(t)(u^*(t))^2 \mathrm dt\nonumber\\
& ~~~~+ \underbrace{\int_{I_{\mathrm{ctrl}}} \mathbb E\!\left[\mathbf 1_{\{t+1/\lambda\in I_{\mathrm{obs}}\}}D\!\left(t+\tfrac1\lambda\right)^2\right]-q(t)(u^*(t))^2}_{H(u^*)}   \mathrm dt\nonumber\\ 
&=H(u^*)
+
\int_{I_{\mathrm{ctrl}}}
q(t)\,|u(t)-u^*(t)|^2\,\mathrm dt ,
\label{eq:globalVarianceBias}
\end{align}
where we have used Equation \eqref{eq:continuousOptControlGeneral} for $u^*(t)$.
Thus, the stochastic optimum separates the irreducible part of the objective from the control-induced error. In particular, the excess cost of any admissible control is exactly its squared distance from $u^*$ in the weighted space $L^2_q$.

For the analysis of the uncertainty itself, we now restrict attention to the interior interval
\[
I_{\mathrm{int}}
=
\left[
\frac{1}{\lambda_{\min}}-\frac{1}{\lambda_{\max}},
\,
T-\frac{1}{\lambda_{\min}}
\right],
\]
where $q(t)=1$ and all velocity realizations contribute. We assume throughout this section that $I_{\mathrm{int}}$ is nonempty and introduce
\begin{align*}
H_{\mathrm{int}}(u)
:=
\int_{I_{\mathrm{int}}}
\mathbb E\!\left[
\left(
u(t)-D\!\left(t+\frac{1}{\lambda}\right)
\right)^2
\right]\,\mathrm dt .
\end{align*}
On this interval, an argument analogous to that used in \eqref{eq:globalVarianceBias} yields
\begin{align}
H_{\mathrm{int}}(u)
=
E_{\mathrm{int}}
+
\|u-u^*\|_{L^2(I_{\mathrm{int}})}^2,
\label{eq:interiorVarianceBias}
\end{align}
where
\begin{align}
E_{\mathrm{int}}
:=
H_{\mathrm{int}}(u^*)
=
\int_{I_{\mathrm{int}}}
\operatorname{Var}\!\left(
D\!\left(t+\frac{1}{\lambda}\right)
\right)\,\mathrm dt .
\label{eq:intrinsicError}
\end{align}

Equation~\eqref{eq:interiorVarianceBias} is the basic identity used below. It shows that all non-optimal control strategies can be compared directly through their $L^2$ distance from the stochastic optimum.

\subsection{Demand uncertainty and velocity uncertainty}

The irreducible error \eqref{eq:intrinsicError} contains two conceptually different contributions. Let
\[
v(s):=\operatorname{Var}(D(s)).
\]
Using the law of total variance with respect to $\lambda$ and assumption~(A5), we obtain
\begin{align*}
\operatorname{Var}\!\left(
D\!\left(t+\frac{1}{\lambda}\right)
\right)
&=
\mathbb E_\lambda\!\left[
v\!\left(t+\frac{1}{\lambda}\right)
\right]
+
\operatorname{Var}_\lambda\!\left(
m\!\left(t+\frac{1}{\lambda}\right)
\right).
\end{align*}
Accordingly,
\begin{align*}
E_{\mathrm{int}}
=
E_{\mathrm{dem}}
+
E_{\mathrm{vel}},
\end{align*}
with
\begin{align*}
E_{\mathrm{dem}}
&:=
\int_{I_{\mathrm{int}}}
\mathbb E_\lambda\!\left[
v\!\left(t+\frac{1}{\lambda}\right)
\right]\,\mathrm dt,
\qquad
E_{\mathrm{vel}}
:=
\int_{I_{\mathrm{int}}}
\operatorname{Var}_\lambda\!\left(
m\!\left(t+\frac{1}{\lambda}\right)
\right)\,\mathrm dt.
%\label{eq:velocityError}
\end{align*}
The first term represents the intrinsic randomness of the demand process. The second term is generated solely by the uncertain travel time: even if the demand were deterministic, different velocity realizations would sample the demand profile at different future times.

\begin{lemma}[Effect of velocity uncertainty]
\label{lem:velocityVariance}
Assume \emph{(A1)--(A5)} and let $m$ be Lipschitz continuous with Lipschitz constant $L_m$. Then
\begin{align}
E_{\mathrm{vel}}
\leq
|I_{\mathrm{int}}|\,L_m^2
\operatorname{Var}\!\left(\frac{1}{\lambda}\right)
\leq
\frac{|I_{\mathrm{int}}|\,L_m^2}{\lambda_{\min}^4}
\operatorname{Var}(\lambda).
\label{eq:velocityVarianceBound}
\end{align}
Consequently,
\[
E_{\mathrm{int}}
=
E_{\mathrm{dem}}
+
O(\operatorname{Var}(\lambda))
\]
as $\operatorname{Var}(\lambda)\to0$, provided the support of $\lambda$ remains in the fixed interval $[\lambda_{\min},\lambda_{\max}]$.
\end{lemma}

\begin{proof}
For any Lipschitz function $f$ and square-integrable random variable $X$,
\[
\operatorname{Var}(f(X))
\leq
\operatorname{Lip}(f)^2\operatorname{Var}(X).
\]
For fixed $t$, the map $r\mapsto m(t+r)$ has Lipschitz constant $L_m$, hence
\[
\operatorname{Var}_\lambda\!\left(
m\!\left(t+\frac{1}{\lambda}\right)
\right)
\leq
L_m^2\operatorname{Var}\!\left(\frac{1}{\lambda}\right).
\]
Moreover, $x\mapsto1/x$ is Lipschitz on $[\lambda_{\min},\lambda_{\max}]$ with constant $\lambda_{\min}^{-2}$. Therefore
\[
\operatorname{Var}\!\left(\frac{1}{\lambda}\right)
\leq
\lambda_{\min}^{-4}\operatorname{Var}(\lambda).
\]
Integration over $I_{\mathrm{int}}$ proves the claim.
\end{proof}

\subsection{Error of the mean-velocity proxy}

On the common interior interval, the stochastic optimum and the deterministic proxy are given by Equation \eqref{eq:continuousOptCOntrol} and \eqref{eq:ProxyControlContinuous}.
% \[
% u^*(t)
% =
% \mathbb E_\lambda\!\left[
% m\!\left(t+\frac{1}{\lambda}\right)
% \right],
% \qquad
% \bar u(t)
% =
% m\!\left(t+\frac{1}{\mathbb E[\lambda]}\right).
% \]
By \eqref{eq:interiorVarianceBias}, the performance loss of the proxy has the exact representation
\begin{align*}
H_{\mathrm{int}}(\bar u)-H_{\mathrm{int}}(u^*)
=
\int_{I_{\mathrm{int}}}
\left|
\mathbb E_\lambda\!\left[
m\!\left(t+\frac{1}{\lambda}\right)
\right]
-
m\!\left(t+\frac{1}{\mathbb E[\lambda]}\right)
\right|^2\,\mathrm dt.
\end{align*}

\begin{lemma}[Mean-velocity proxy error]
\label{lem:proxyError}
Assume \emph{(A1)--(A5)} and $m\in C^2([0,T])$. Then there exists a constant $C>0$, depending only on $\lambda_{\min}$, $\lambda_{\max}$, $|I_{\mathrm{int}}|$, and $\|m\|_{C^2}$, such that
\begin{align}
0
\leq
H_{\mathrm{int}}(\bar u)-H_{\mathrm{int}}(u^*)
\leq
C\,\operatorname{Var}(\lambda)^2.
\label{eq:proxyVarianceSquared}
\end{align}
Consequently,
\begin{align}
H_{\mathrm{int}}(\bar u)
=
E_{\mathrm{dem}}
+
O(\operatorname{Var}(\lambda))
+
O(\operatorname{Var}(\lambda)^2).
\label{eq:proxyTotalError}
\end{align}
\end{lemma}

\begin{proof}
For fixed $t$, define
\[
h_t(\ell):=m\!\left(t+\frac{1}{\ell}\right).
\]
Since $\ell\in[\lambda_{\min},\lambda_{\max}]$ and $m\in C^2([0,T])$, the second derivative of $h_t$ is uniformly bounded in $t$ and $\ell$. Taylor expansion around $\bar\lambda=\mathbb E[\lambda]$ gives
\[
h_t(\lambda)
=
h_t(\bar\lambda)
+
h_t'(\bar\lambda)(\lambda-\bar\lambda)
+
\frac12 h_t''(\xi_{t,\lambda})(\lambda-\bar\lambda)^2
\]
for a suitable intermediate point $\xi_{t,\lambda}$. Taking expectations eliminates the linear term and yields, for a constant $C_0>0$,
\[
\left|
\mathbb E[h_t(\lambda)]-h_t(\bar\lambda)
\right|
\leq
C_0\operatorname{Var}(\lambda)
\]
uniformly in $t\in I_{\mathrm{int}}$. Squaring and integrating proves \eqref{eq:proxyVarianceSquared}. Equation~\eqref{eq:proxyTotalError} follows from Lemma~\ref{lem:velocityVariance}.
\end{proof}

This distinction is important: velocity uncertainty contributes at order $O(\operatorname{Var}(\lambda))$ to the irreducible stochastic error, whereas the \emph{additional} loss caused by replacing the random velocity with its mean is of order $O(\operatorname{Var}(\lambda)^2)$.

\subsection{Stability with respect to the velocity distribution}
\label{sec:distributionStability}

The mean-velocity proxy studied above is one particular approximation of the random transport dynamics. More generally, in applications the law of the transport velocity may only be known approximately, for instance because it has to be estimated from data. We therefore investigate how sensitively the optimal inflow depends on the probability distribution of the velocity.

Let $\mu$ and $\nu$ be two probability measures supported on the common interval $[\lambda_{\min},\lambda_{\max}]$. On a fixed compact interval $J\Subset I_{\mathrm{int}}$, define the corresponding optimal controls by
\begin{align}
\begin{split}
    u_\mu^*(t)
&:=
\int_{\lambda_{\min}}^{\lambda_{\max}}
m\!\left(t+\frac{1}{\ell}\right)\,\mu(\mathrm d\ell),
\\
u_\nu^*(t)
&:=
\int_{\lambda_{\min}}^{\lambda_{\max}}
m\!\left(t+\frac{1}{\ell}\right)\,\nu(\mathrm d\ell).
\label{eq:controlsTwoVelocityLaws}
\end{split}
\end{align}
We denote by $W_1$ the Wasserstein--1 distance on probability measures over $[\lambda_{\min},\lambda_{\max}]$.

\begin{theorem}[Stability with respect to the velocity law]
\label{thm:velocityLawStability}
Assume that $m$ is Lipschitz continuous with Lipschitz constant $L_m$. Then
\begin{align}
\|u_\mu^*-u_\nu^*\|_{L^\infty(J)}
&\leq
\frac{L_m}{\lambda_{\min}^2}
W_1(\mu,\nu),
\label{eq:stabilityLinfty}
\\
\|u_\mu^*-u_\nu^*\|_{L^2(J)}
&\leq
\frac{L_m\sqrt{|J|}}{\lambda_{\min}^2}
W_1(\mu,\nu).
\label{eq:stabilityL2}
\end{align}
In particular, the optimal control depends Lipschitz continuously on the velocity distribution with respect to the Wasserstein--1 distance.
\end{theorem}

\begin{proof}
Fix $t\in J$ and define
\[
h_t(\ell)
:=
m\!\left(t+\frac{1}{\ell}\right).
\]
Since $m$ is Lipschitz continuous and the map $\ell\mapsto1/\ell$ is Lipschitz on $[\lambda_{\min},\lambda_{\max}]$ with constant $\lambda_{\min}^{-2}$, one has
\[
|h_t(\ell_1)-h_t(\ell_2)|
\leq
\frac{L_m}{\lambda_{\min}^2}
|\ell_1-\ell_2|.
\]
Hence $h_t$ is Lipschitz with
\[
\operatorname{Lip}(h_t)
\leq
\frac{L_m}{\lambda_{\min}^2},
\]
uniformly in $t\in J$. By the Kantorovich--Rubinstein dual representation of $W_1$,
\[
\left|
\int h_t\,\mathrm d\mu
-
\int h_t\,\mathrm d\nu
\right|
\leq
\operatorname{Lip}(h_t)\,W_1(\mu,\nu).
\]
Using \eqref{eq:controlsTwoVelocityLaws} gives
\[
|u_\mu^*(t)-u_\nu^*(t)|
\leq
\frac{L_m}{\lambda_{\min}^2}
W_1(\mu,\nu).
\]
Taking the supremum over $t\in J$ proves \eqref{eq:stabilityLinfty}. The $L^2$ estimate \eqref{eq:stabilityL2} follows immediately from

\[
\|f\|_{L^2(J)}
\leq
\sqrt{|J|}\,\|f\|_{L^\infty(J)}.
\qedhere
\]
\end{proof}

The theorem has a direct robustness interpretation: if the velocity law is estimated with an error of size $W_1(\mu,\nu)$, then the resulting optimal inflow changes by at most the same order. This stability result is independent of the particular mean-velocity approximation considered above.

The same estimate also yields a performance bound when a control optimized for the approximate law $\nu$ is applied to the system governed by the true law $\mu$. Let
\[
H_{\mu,J}(u)
:=
\int_J
\mathbb E_\mu\!\left[
\left(
u(t)-D\!\left(t+\frac{1}{\lambda}\right)
\right)^2
\right]\,\mathrm dt
\]
denote the interior objective associated with $\mu$.

\begin{corollary}[Performance under distributional misspecification]
\label{cor:velocityLawPerformance}
Under the assumptions of Theorem~\ref{thm:velocityLawStability},
\begin{align*}
0
\leq
H_{\mu,J}(u_\nu^*)
-
H_{\mu,J}(u_\mu^*)
\leq
\frac{L_m^2|J|}{\lambda_{\min}^4}
W_1(\mu,\nu)^2.
\end{align*}
\end{corollary}

\begin{proof}
Using the variance--bias decomposition for the objective corresponding to $\mu$,
\[
H_{\mu,J}(u)
=
H_{\mu,J}(u_\mu^*)
+
\|u-u_\mu^*\|_{L^2(J)}^2.
\]
Setting $u=u_\nu^*$ and applying \eqref{eq:stabilityL2} yields the result.
\end{proof}

\begin{remark}[Random travel-time laws]
\label{rem:travelTimeLawStability}
The preceding argument depends only on the random travel time $\tau=1/\lambda$. If $\eta$ and $\zeta$ are two probability laws for a positive random travel time $\tau$ and
\[
u_\eta^*(t)
=
\int m(t+r)\,\eta(\mathrm dr),
\qquad
u_\zeta^*(t)
=
\int m(t+r)\,\zeta(\mathrm dr),
\]
then
\[
\|u_\eta^*-u_\zeta^*\|_{L^2(J)}
\leq
L_m\sqrt{|J|}\,W_1(\eta,\zeta).
\]
Thus the stability mechanism is not specific to the representation $\tau=1/\lambda$ and extends directly to general random transport delays.
\end{remark}

\subsection{Error for the discrete controls}

We next quantify the additional error caused by restricting the control to piecewise constant functions. Let $u^{(n),*}$ denote the optimal piecewise constant control from Proposition~\ref{prop:optimalDiscreteControl}. The global identity \eqref{eq:globalVarianceBias} gives
\begin{align*}
H(u^{(n),*})-H(u^*)
=
\int_{I_{\mathrm{ctrl}}}
q(t)\,
|u^{(n),*}(t)-u^*(t)|^2\,\mathrm dt.
\end{align*}
Thus, no additional probabilistic distance is required: the discretization error is exactly the weighted projection error.

\begin{lemma}[Piecewise constant discretization error]
\label{lem:discretizationError}
Assume \emph{(A1)--(A5)} and let $J\Subset I_{\mathrm{int}}$. Suppose that $m$ is Lipschitz continuous and that the control partitions resolve the endpoints of $J$. Then 
\begin{align}
H_J(u^{(n),*})
=
H_J(u^*)
+
O(|\Pi_n|^2).
\label{eq:discreteInteriorRate}
\end{align}
If, in addition, $u^*$ is globally Lipschitz on $I_{\mathrm{ctrl}}$, then the same rate holds for the full objective:
\begin{align}
H(u^{(n),*})
=
H(u^*)
+
O(|\Pi_n|^2).
\label{eq:discreteObjectiveRate}
\end{align}
\end{lemma}

\begin{proof}
On $J\subset I_{\mathrm{int}}$, one has $q\equiv1$, and Lipschitz continuity of $m$ implies Lipschitz continuity of
\[
u^*(t)=\mathbb E_\lambda\!\left[m\!\left(t+\frac1\lambda\right)\right].
\]
The standard cellwise Poincar\'e estimate therefore yields
\[
\int_J |u^{(n),*}(t)-u^*(t)|^2\,\mathrm dt
\le C|\Pi_n|^2\|(u^*)'\|_{L^2(J)}^2.
\]
Together with the variance--bias identity, this proves \eqref{eq:discreteInteriorRate}.

For the global statement, assume that $u^*$ from \eqref{eq:continuousOptControlGeneral} is Lipschitz with constant $L_*$. Note that this is not trivial from $m$ being Lipschitz as this property also depends on $q(t)$ and $\Lambda(t)$. On a cell $I_i=[t_i,t_{i+1})$, the weighted coefficient $u_i^{(n),*}$ minimizes
\[
c\mapsto\int_{I_i}q(t)|u^*(t)-c|^2\,\mathrm dt.
\]
Choosing, for instance, $c=u^*(t_i)$ gives
\[
\int_{I_i}q(t)|u^*(t)-u_i^{(n),*}|^2\,\mathrm dt
\le
L_*^2 |I_i|^2\int_{I_i}q(t)\,\mathrm dt.
\]
Summing over all cells yields
\[
H(u^{(n),*})-H(u^*)
\le
L_*^2|\Pi_n|^2\int_{I_{\mathrm{ctrl}}}q(t)\,\mathrm dt,
\]
which proves \eqref{eq:discreteObjectiveRate}.
\end{proof}

Compared with the previous formulation, the role of temporal discretization is now explicit: it is a deterministic projection error and is independent of the probabilistic mechanism generating the random travel time.

\subsection{Piecewise constant proxy control}

Finally, let $\bar u^{(n)}$ be the cell-average approximation of the mean-velocity proxy on a partition of $I_{\mathrm{int}}$,
\[
\bar u_i^{(n)}
=
\frac{1}{t_{i+1}-t_i}
\int_{t_i}^{t_{i+1}}
m\!\left(t+\frac{1}{\mathbb E[\lambda]}\right)\,\mathrm dt.
\]
The total excess error can be separated into the proxy error and the temporal discretization error.

\begin{proposition}[Piecewise constant proxy error]
\label{prop:piecewiseProxyError}
Assume \emph{(A1)--(A5)} and $m\in C^2([0,T])$. Then
\begin{align*}
H_{\mathrm{int}}(\bar u^{(n)})
=
E_{\mathrm{int}}
+
O(\operatorname{Var}(\lambda)^2)
+
O(|\Pi_n|^2).
\end{align*}
Together with Lemma~\ref{lem:velocityVariance}, this yields
\begin{align*}
H_{\mathrm{int}}(\bar u^{(n)})
=
E_{\mathrm{dem}}
+
O(\operatorname{Var}(\lambda))
+
O(|\Pi_n|^2).
\end{align*}
\end{proposition}

\begin{proof}
Using \eqref{eq:interiorVarianceBias},
\[
H_{\mathrm{int}}(\bar u^{(n)})-E_{\mathrm{int}}
=
\|\bar u^{(n)}-u^*\|_{L^2(I_{\mathrm{int}})}^2.
\]
Insert $\bar u$ and use
\[
\|\bar u^{(n)}-u^*\|_{L^2}^2
\leq
2\|\bar u-u^*\|_{L^2}^2
+
2\|\bar u^{(n)}-\bar u\|_{L^2}^2.
\]
The first term is $O(\operatorname{Var}(\lambda)^2)$ by Lemma~\ref{lem:proxyError}, while the second is $O(|\Pi_n|^2)$ by the standard piecewise constant approximation estimate.
\end{proof}

\begin{remark}[Discrete proxy comparison]
\label{rem:discreteProxyComparison}
On a fixed interval $J\Subset I_{\mathrm{int}}$, let $P_n$ denote the $L^2(J)$-orthogonal projection onto the piecewise constant control space. Since $u^{(n),*}=P_nu^*$ and $\bar u^{(n)}=P_n\bar u$ on $J$, the contraction property of $P_n$ gives
\[
H_J(\bar u^{(n)})-H_J(u^{(n),*})
=
\|P_n(\bar u-u^*)\|_{L^2(J)}^2
\leq
\|\bar u-u^*\|_{L^2(J)}^2.
\]
Hence Lemma~\ref{lem:proxyError} also yields
\[
H_J(\bar u^{(n)})-H_J(u^{(n),*})
=O(\operatorname{Var}(\lambda)^2).
\]
Thus the second-order proxy estimate applies directly to the piecewise constant comparison used in the numerical experiments below.
\end{remark}

% The resulting hierarchy of errors is therefore:
% \[
% \boxed{
% \text{demand uncertainty}
% \;+\;
% \text{random-travel-time uncertainty}
% \;+\;
% \text{proxy approximation}
% \;+\;
% \text{temporal discretization}.
% }
% \]
% The first two contributions are intrinsic to the stochastic system, whereas the latter two are introduced by the chosen control strategy. 
In particular, the leading-order influence of velocity uncertainty is $O(\operatorname{Var}(\lambda))$, while the additional mean-velocity proxy loss is of higher order, $O(\operatorname{Var}(\lambda)^2)$.

\section{Summary of the analytical results}
\label{sec:summaryControls}

The preceding analysis distinguishes two modeling choices: the stochastic transport model and its deterministic mean-velocity proxy and two implementation classes, i.e., continuous-time and piecewise constant controls. Table~\ref{tab:controlSummary} collects the corresponding control laws on the interior interval $I_{\mathrm{int}}$, where all velocity realizations contribute and temporal boundary effects are absent.
% \begin{table}[ht]
% \centering
% \caption{Control laws and excess costs on the interior interval $I_{\mathrm{int}}$. Here $m(t)=\mathbb E[D(t)]$, $\bar\lambda=\mathbb E[\lambda]$, and $E_{\mathrm{int}}=H_{\mathrm{int}}(u^*)$.}
% \label{tab:controlSummary}
% \renewcommand{\arraystretch}{1.45}
% \begin{tabular}{p{0.20\textwidth} p{0.35\textwidth} p{0.34\textwidth}}
% \hline
% \textbf{Control strategy}
% &
% \textbf{Control law}
% &
% \textbf{Objective / excess cost}
% \\
% \hline
% Stochastic, continuous
% &
% $\displaystyle
% u^*(t)
% =
% \mathbb E_\lambda\!\left[
% m\!\left(t+\frac{1}{\lambda}\right)
% \right]
% $
% &
% $\displaystyle
% H_{\mathrm{int}}(u^*)=E_{\mathrm{int}}
% $
% \\[2mm]

% Stochastic, piecewise constant
% &
% $\displaystyle
% u_i^{(n),*}
% =
% \frac{1}{t_{i+1}-t_i}
% \int_{t_i}^{t_{i+1}}u^*(t)\,\mathrm dt
% $
% &
% $\displaystyle
% H_{\mathrm{int}}(u^{(n),*})
% =
% E_{\mathrm{int}}
% +
% O(|\Pi_n|^2)
% $
% \\[2mm]

% Mean-velocity proxy, continuous
% &
% $\displaystyle
% \bar u(t)
% =
% m\!\left(
% t+\frac{1}{\bar\lambda}
% \right)
% $
% &
% $\displaystyle
% H_{\mathrm{int}}(\bar u)
% =
% E_{\mathrm{int}}
% +
% O(\operatorname{Var}(\lambda)^2)
% $
% \\[2mm]

% Mean-velocity proxy, piecewise constant
% &
% $\displaystyle
% \bar u_i^{(n)}=
% \frac{1}{t_{i+1}-t_i}
% \int_{t_i}^{t_{i+1}}
% m\!\left(
% t+\frac{1}{\bar\lambda}
% \right)\,\mathrm dt
% $
% &
% $\displaystyle
% H_{\mathrm{int}}(\bar u^{(n)})
% =
% E_{\mathrm{int}}
% +
% O(\operatorname{Var}(\lambda)^2)
% +
% O(|\Pi_n|^2)
% $
% \\
% \hline
% \end{tabular}
% \end{table}

\begin{table}[ht]
\centering
\caption{Control laws and excess costs on the interior interval $I_{\mathrm{int}}$. Here $m(t)=\mathbb E[D(t)]$, $\bar\lambda=\mathbb E[\lambda]$, and $E_{\mathrm{int}}=H_{\mathrm{int}}(u^*)$.}
\label{tab:controlSummary}
\renewcommand{\arraystretch}{1.5}
\setlength{\tabcolsep}{6pt}
\rowcolors{2}{gray!8}{white}
\begin{tabular}{p{0.20\textwidth} p{0.35\textwidth} p{0.34\textwidth}}
\toprule
\textbf{Control}\newline \textbf{strategy} & \textbf{Optimal control} & \textbf{Objective / excess cost} \\
\midrule
Stochastic,\newline continuous
&
$\displaystyle u^*(t) = \mathbb E_\lambda\!\left[m\!\left(t+\frac{1}{\lambda}\right)\right]$\newline
{\footnotesize\color{gray!70!black} Eq.~\eqref{eq:continuousOptCOntrol}}
&
$\displaystyle H_{\mathrm{int}}(u^*)=E_{\mathrm{int}}\newline \phantom{quandm}= E_\text{dem} + O(\operatorname{Var}(\lambda))$ \newline
{\footnotesize\color{gray!70!black} Eq.~\eqref{eq:intrinsicError}, Lemma \ref{lem:velocityVariance}}
\\[3mm]
Stochastic, piecewise constant
&
$\displaystyle u_i^{(n),*} = \frac{1}{t_{i+1}-t_i}\int_{t_i}^{t_{i+1}}u^*(t)\,\mathrm dt$
\newline
{\footnotesize\color{gray!70!black} Eq.~\eqref{eq:pwConstantInterior}}
&
$\displaystyle H_{\mathrm{int}}(u^{(n),*}) = E_{\mathrm{int}} + O(|\Pi_n|^2)$ \newline
{\footnotesize\color{gray!70!black} Lemma~\ref{lem:discretizationError}}
\\[3mm]
Mean-velocity proxy, continuous
&
$\displaystyle \bar u(t) = m\!\left(t+\frac{1}{\bar\lambda}\right)$\newline
{\footnotesize\color{gray!70!black} Eq.~\eqref{eq:ProxyControlContinuous}}
&
$\displaystyle H_{\mathrm{int}}(\bar u) = E_{\mathrm{int}} + O(\operatorname{Var}(\lambda)^2)$ \newline
{\footnotesize\color{gray!70!black} Lemma~\ref{lem:proxyError}}
\\[3mm]
Mean-velocity proxy, piecewise constant
&
$\displaystyle \bar u_i^{(n)}=\newline\frac{1}{t_{i+1}-t_i}\int_{t_i}^{t_{i+1}} m\!\left(t+\frac{1}{\bar\lambda}\right)\,\mathrm dt$\newline
{\footnotesize\color{gray!70!black} Eq.~\eqref{eq:ProxyControlPWConstant}}
&
$\displaystyle H_{\mathrm{int}}(\bar u^{(n)}) \newline= E_{\mathrm{int}} + O(\operatorname{Var}(\lambda)^2) + O(|\Pi_n|^2)$ \newline
{\footnotesize\color{gray!70!black} Prop.~\ref{prop:piecewiseProxyError}}
\\
\bottomrule
\end{tabular}
\end{table}

The minimal stochastic error itself decomposes as
\begin{align*}
E_{\mathrm{int}}
=
E_{\mathrm{dem}}
+
E_{\mathrm{vel}},
\qquad
E_{\mathrm{vel}}
=
O(\operatorname{Var}(\lambda)),
\end{align*}
where $E_{\mathrm{dem}}$ is generated by the intrinsic randomness of the demand and $E_{\mathrm{vel}}$ by the random travel time. Consequently, the leading-order effect of velocity uncertainty is already present in the optimally controlled stochastic system. By contrast, replacing the random velocity by its mean introduces only the higher-order excess cost
\[
H_{\mathrm{int}}(\bar u)-H_{\mathrm{int}}(u^*)
=
O(\operatorname{Var}(\lambda)^2).
\]
Temporal discretization produces an additional $O(|\Pi_n|^2)$ contribution to the objective.

For the quadratic discretization rate, the rigorous statement is made on a fixed compact interval $J\Subset I_{\mathrm{int}}$; see Lemma~\ref{lem:discretizationError}. A global rate requires the additional assumption that $u^*$ is Lipschitz on $I_{\mathrm{ctrl}}$.

The table deliberately reports the simple interior formulas. On the full control interval $I_{\mathrm{ctrl}}$, temporal boundary effects have to be taken into account. The stochastic optimum is then the conditional average
\[
u^*(t)
=
\frac{
\displaystyle
\int_{\Lambda(t)}
m\!\left(t+\frac{1}{\ell}\right)
\phi_\lambda(\ell)\,\mathrm d\ell
}{
\displaystyle
%\int_{\Lambda(t)}\phi_\lambda(\ell)\,\mathrm d\ell
q(t)
},
\]
and the piecewise constant stochastic control is its $q$-weighted cell average; see Theorem~\ref{thm:continuousOptimalControl} and Proposition~\ref{prop:optimalDiscreteControl}. Thus, the interior formulas in Table~\ref{tab:controlSummary} describe the genuine uncertainty mechanism, whereas the additional weights near the temporal boundaries originate solely from the fixed downstream observation window.

This separation will also guide the numerical experiments. We first verify the predicted dependence on the velocity variance and the temporal mesh size, and then assess how well the mean-velocity proxy performs beyond the linear setting.

\section{Numerical analysis}
\label{sec:numericalAnalysis}

The numerical experiments are organized around the error mechanisms identified in Sections~\ref{sec:errorAnalysis} and \ref{sec:summaryControls}. We first study temporal control discretization and verify the convergence of piecewise constant controls. We then decrease the variance of the random transport velocity in order to examine the intrinsic random-travel-time error and the accuracy of the mean-velocity proxy. Finally, we consider a nonlinear nonlocal transport model for which no explicit optimal control formula is available and investigate whether the proxy strategy remains effective beyond the linear theory.

\subsection{Numerical setup}

For the linear model, the transport equation \eqref{eq:stochasticPDE} is discretized by a first-order upwind method. We use a time step $\Delta t=5\cdot10^{-4}$ and, for each realization of the velocity, choose $\Delta x=\lambda(\omega)\Delta t$, so that the CFL number equals one. The last spatial cell is adjusted to match the endpoint $x=1$. With this choice, the numerical propagation follows the characteristic transport as closely as possible and avoids an additional CFL-induced smearing of the effects studied below.

The uncertain demand is modeled by a time-inhomogeneous Jacobi process with state space $[0,4]$,
\begin{align}
\mathrm dD(t)
=
\kappa\bigl(\theta(t)-D(t)\bigr)\,\mathrm dt
+
\sigma\sqrt{(D(t)-a)(b-D(t))}\,\mathrm dW(t),
\label{eq:JacobiSDE}
\end{align}
with
\[
a=0,\qquad b=4,\qquad D(0)=1.6,\qquad
\kappa=4,\qquad
\theta(t)=2+\sin(\pi t),\qquad
\sigma=0.15.
\]
We approximate \eqref{eq:JacobiSDE} by the Euler--Maruyama method with time step $\Delta t=10^{-3}$ and project every update onto $[0,4]$. This is the only point in the paper where a particular stochastic demand model is required; the analytical results of the previous sections use only the first two demand moments. To estimate mean outflow and demand quantities, we use a Monte Carlo simulation with 25,000 runs.

To separate the asymptotic effects from the irreducible demand variance, we report not only the objective values themselves but, where appropriate, also the corresponding \emph{excess costs}. In particular, the temporal discretization experiment is compared through
\[
H_J(u^{(n),*})-H_J(u^*),
\]
whereas the proxy experiment is compared through
\[
H_{J}(\bar u)-H_{J}(u^*).
\]
Here, $J$ denotes a suitably chosen compact interior interval. These are precisely the quantities for which Section~\ref{sec:errorAnalysis} predicts second-order behavior.

\subsection{Temporal discretization: convergence of piecewise constant controls}
\label{sec:numTimeDiscretization}
We first fix the velocity distribution to
\[
\lambda\sim\mathcal U([1,3])
\]
and take $T=16$ and $J=\left[2,\,T-2\right]$. The control grid is equidistant with cell length
\[
|\Pi_n|=2^k,\qquad k=-4,\ldots,1.
\]
Lemma~\ref{lem:discretizationError} predicts on this fixed interior interval
\[
H_J(u^{(n),*})-H_J(u^*)=O(|\Pi_n|^2).
\]

% Figure~\ref{fig:pwcLimit_error} shows the total-objective plot and contains the non-vanishing intrinsic stochastic error. Hence its log--log slope is not expected to remain equal to two once this background contribution dominates. We show not only the classical stochastic velocity framework and the piecewise constant stochastic control framework, but also a version where also in the forward problem the velocity is deterministic like in \cite{Goettlich2022}. As expected the latter error is the smallest, as uncertainty enters exclusively by the uncertain demand. The classical stochastic continuous-time problem shows a slightly lower error than using the proxy-control which is also expected.
% \begin{figure}[ht!]
%     \centering
%     \includegraphics[width=0.75\linewidth]{figures_25_05_29/error_PWC_withNoise.pdf}
%     \caption{Numerical result for decreasing lengths of the piecewise constant control intervals under uncertain demand. }
%     \label{fig:pwcLimit_error}
% \end{figure}
Figure~\ref{fig:pwcLimit_inflowOutflow} illustrates the expected convergence of the piecewise constant controls toward the continuous-time strategy. The left panel shows how, starting from a constant inflow level, the control increasingly resolves the sinusoidal pattern of the mean demand as the temporal grid is refined. The right panel compares the corresponding outflows with the mean demand. The outflows are temporally shifted and averaged versions of the controls induced by the random transport dynamics, and they approach a sinusoidal profile with the same periodicity as the expected demand. The remaining mismatch between mean outflow and mean demand is caused by the intrinsic stochastic effect of the uncertain velocity.

\begin{figure}[htbp]
    \centering
    \begin{minipage}{0.48\textwidth}
        \centering
        \includegraphics[width=\linewidth]{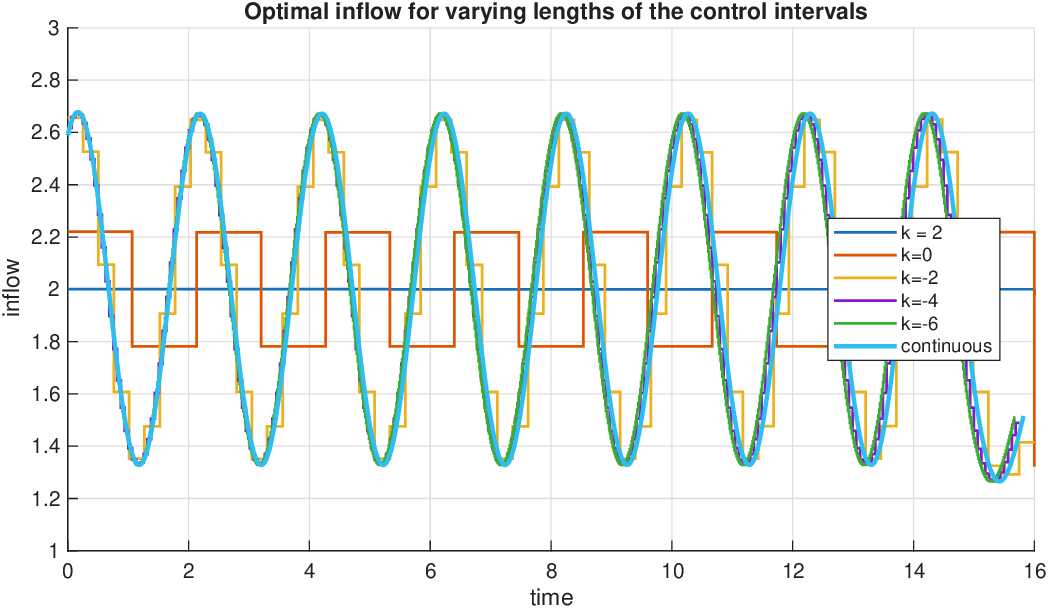}
    \end{minipage}
    \hfill
    \begin{minipage}{0.48\textwidth}
        \centering
        \includegraphics[width=\linewidth]{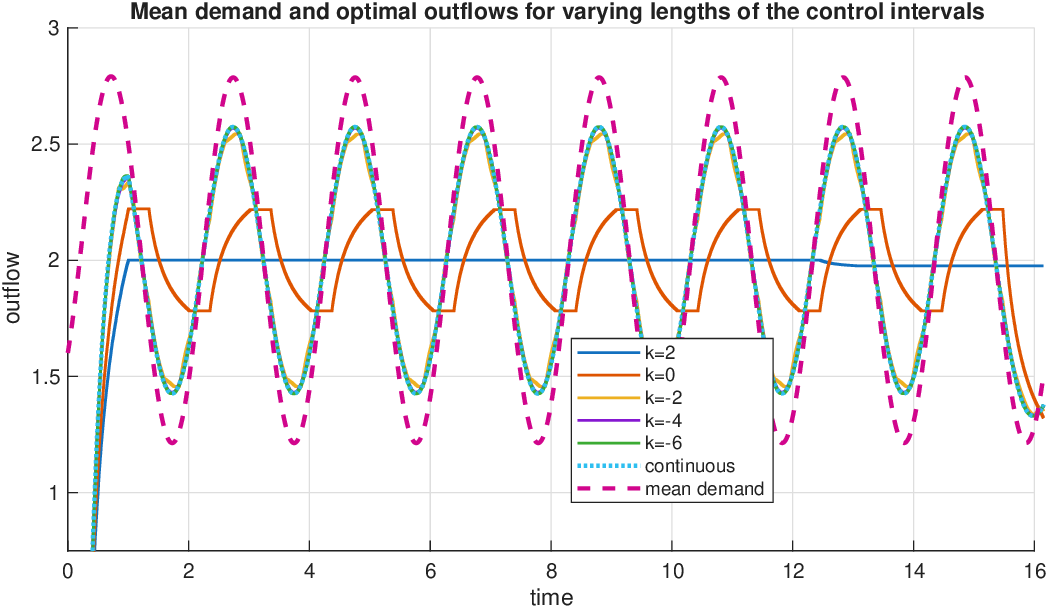}
    \end{minipage}
    \caption{Optimal inflows (left), outflows, and mean demand (right) for decreasing lengths of the piecewise constant control intervals.}
    \label{fig:pwcLimit_inflowOutflow}
\end{figure}

The theoretically relevant convergence quantity is the excess cost $H_J(u^{(n),*})-H_J(u^*)$, shown in Figure~\ref{fig:numericalConvergenceTime}. The observed numerical convergence rates are very close to the predicted rate of 2 for both the continuous-time and piecewise constant controls. 

\begin{figure}[htbp]
    \centering
    \includegraphics[width=0.6\linewidth]{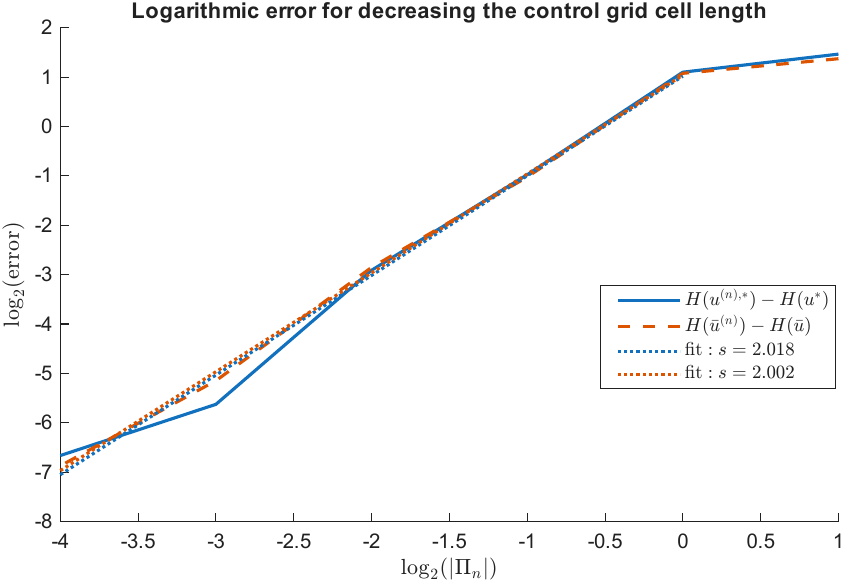}
    \caption{Log--log plot of the excess cost for decreasing piecewise constant control intervals, for continuous-time and piecewise constant inflows.}
    \label{fig:numericalConvergenceTime}
\end{figure}

\subsection{Velocity uncertainty and the mean-velocity proxy}
\label{sec:numVelocityVariance}
We keep the mean velocity fixed at $2$ and reduce the support of the uniform distribution symmetrically,
\[
\lambda_{\min}
=
2-\frac12\sqrt{2^k},
\qquad
\lambda_{\max}
=
2+\frac12\sqrt{2^k},
\qquad k=-6,\ldots,3,
\]
while we now fix $|\Pi_n| = \frac{1}{2}$
Then $\operatorname{Var}(\lambda)$ is proportional to $2^k$. This choice ensures $\frac{1}{2}<\lambda_{\min}$, so that $J=[2,T-2]$ is contained in the interior interval for every velocity distribution considered.

The left panel of Figure~\ref{fig:varLambda_error}, analogous to the right panel of Figure~\ref{fig:pwcLimit_inflowOutflow}, compares the optimal outflows for different velocity variances with the expected demand. As the variance decreases, the amplitude of the expected outflow increases and approaches that of the mean demand.
The right panel of Figure~\ref{fig:varLambda_error} shows the logarithmic error evolution as $\operatorname{Var}(\lambda)\to0$ for continuous-time controls (solid lines) and piecewise constant controls (dashed lines). In addition to the stochastic formulations, we include a reference case in which the velocity is also taken to be deterministic in the forward simulation. This reference exhibits a constant baseline error arising solely from demand uncertainty. As the velocity law concentrates at $\lambda=2$, the stochastic models converge to this deterministic reference, while the corresponding proxy models rapidly approach the stochastic ones.

\begin{figure}[htbp]
    \centering
    \begin{minipage}{0.48\textwidth}
       \centering
        \includegraphics[width=\linewidth]{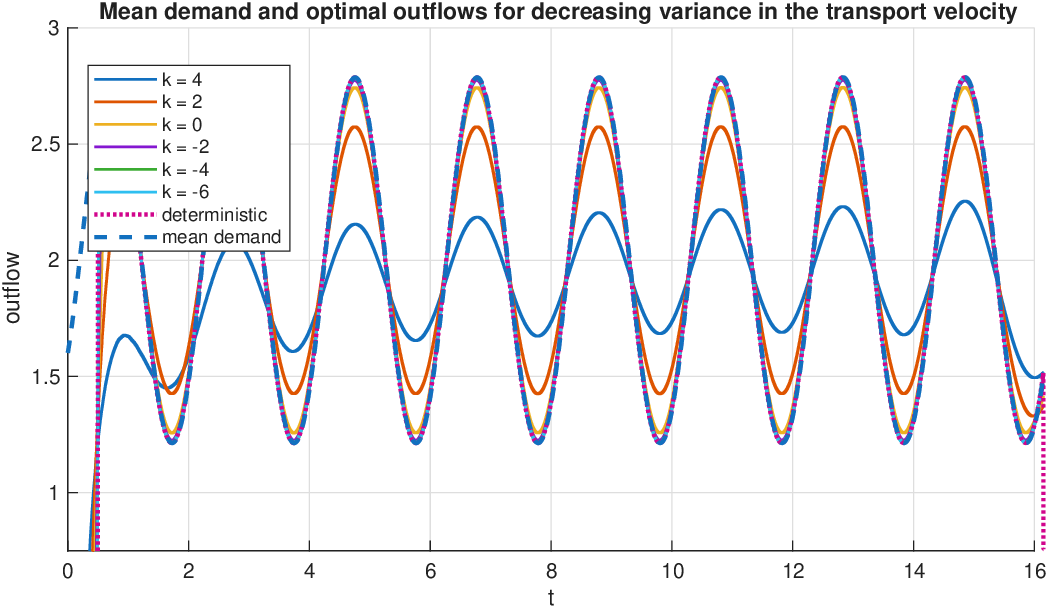}
    \end{minipage}
    \hfill
    \begin{minipage}{0.48\textwidth}
        \centering
        \includegraphics[width=\linewidth]{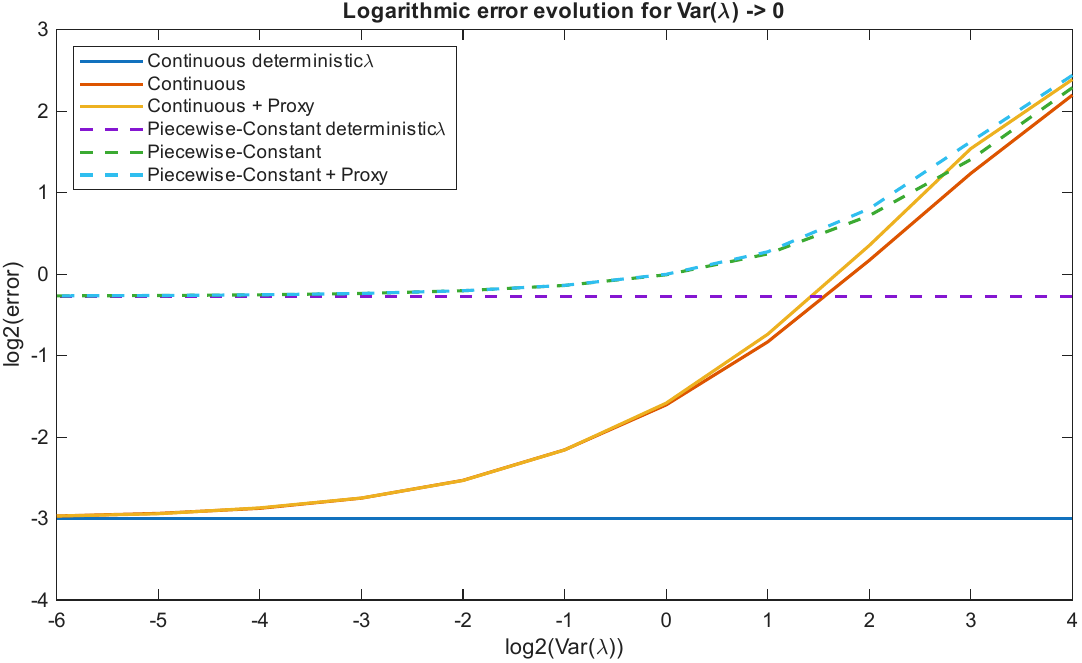}
    \end{minipage}
    \caption{Optimal outflows and mean demand for decreasing variance in the transport velocity (left) and total-objective results for decreasing variance in $\lambda$ (right).}
    \label{fig:varLambda_error}
\end{figure}

This convergence is shown more explicitly in Figure~\ref{fig:numericalConvergenceSpeed}, where we report the excess costs 
\[
H_J(\bar u)-H_J(u^*), \qquad H_J(\bar u^{(n)})-H_J(u^{(n),*}).
\]
The results confirm the predicted second-order behavior; for the piecewise constant controls, the observed numerical rate is slightly above 2. 

\begin{figure}[htbp]
    \centering
    \includegraphics[width=0.6\linewidth]{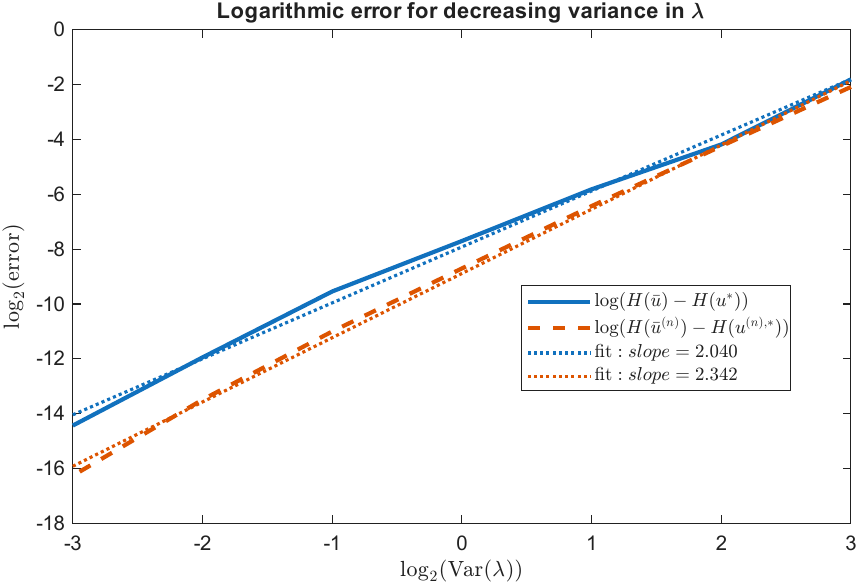}
    \caption{Log--log plot of the excess cost for decreasing transport-velocity variance, for continuous-time and piecewise constant inflows.}
    \label{fig:numericalConvergenceSpeed}
\end{figure}

% The theory distinguishes the intrinsic random-travel-time contribution,
% \[
% E_{\mathrm{vel},J}=O(\operatorname{Var}(\lambda)),
% \]
% from the additional loss of the mean-velocity proxy,
% \[
% H_J(\bar u)-H_J(u^*)=O(\operatorname{Var}(\lambda)^2).
% \]
% The total-objective curve in Figure~\ref{fig:varLambda_error} cannot by itself distinguish these two orders.

% \begin{figure}[htbp]
%     \centering
%     \begin{minipage}{0.48\textwidth}
%         \centering
%         \includegraphics[width=\linewidth]{figures_25_05_29/errorVarLambda_withNoise_control.pdf}
%     \end{minipage}
%     \hfill
%     \begin{minipage}{0.48\textwidth}
%         \centering
%         \includegraphics[width=\linewidth]{figures_25_05_29/errorVarLambda_withNoise_outflowDemand.pdf}
%     \end{minipage}
%     \caption{Existing continuous controls, outflows, and mean demand for decreasing variance in the transport velocity.}
%     \label{fig:Var_contLimit_inflowOutflow}
% \end{figure}

% \begin{figure}[htbp]
%     \centering
%     \begin{minipage}{0.48\textwidth}
%         \centering
%         \includegraphics[width=\linewidth]{figures_25_05_29/errorVarLambda_withNoise_control_piecewiseconstant.pdf}
%     \end{minipage}
%     \hfill
%     \begin{minipage}{0.48\textwidth}
%         \centering
%         \includegraphics[width=\linewidth]{figures_25_05_29/errorVarLambda_withNoise_outflowDemand_piecewiseconstant.pdf}
%     \end{minipage}
%     \caption{Existing piecewise constant controls, outflows, and mean demand for decreasing variance in the transport velocity.}
%     \label{fig:Var_PWCLimit_inflowOutflow}
% \end{figure}

\subsection{Boundary effects of the fixed observation window}
We next illustrate why the control must be adjusted near the temporal boundaries rather than using the shifted expected demand throughout. We take $\lambda\sim\mathcal U([1,3])$ and plot, on the full control interval,
\[
I_{\mathrm{ctrl}}=[0,T-1/3]
\]
both the corrected control
\[
u^*(t)=
\frac{\int_{\Lambda(t)}m(t+1/\ell)\phi_\lambda(\ell)\,\mathrm d\ell}
{\int_{\Lambda(t)}\phi_\lambda(\ell)\,\mathrm d\ell}
\]
and the unconditioned interior expression
\[
\mathbb E_\lambda[m(t+1/\lambda)].
\]
We set $T=6$ to emphasize the boundary region and keep all other parameters unchanged.
The two curves coincide on $I_{\mathrm{int}}=[2/3,T-1]$ and differ near the temporal boundaries. Hence, the left panel of Figure~\ref{fig:boundaryAdjustment} directly illustrates the boundary correction introduced in Section~\ref{sec:speedRV}. Lemma~\ref{lem:discretizationError} establishes second-order convergence for the interior objective and, under the additional assumption that the optimal inflow is globally Lipschitz continuous, also for the full objective including boundary effects. The numerical experiment exhibits a convergence rate close to 2 for the full objective as well. This is consistent with the global estimate of Lemma~\ref{lem:discretizationError} under its additional regularity assumption, but no global rate is claimed here beyond this numerical evidence.

%While Theorem \ref{thm:discreteConvergence} states convergence, there is no statement on the convergence rate, in particular it is not ensured that order 2 is reached as for $J\Subset I_\text{int}$. This is underlined by our simulations where we obtain numerical convergence rates of below 2 but above 1.
% \begin{figure}[htbp]
%     \centering
%     \includegraphics[width=0.6\linewidth]{figures_25_05_29/Figure_154.pdf}
%     \caption{Comparison of the optimal control and the unconditioned expression for the interior parts.}
%     \label{fig:boundaryAdjustment}
% \end{figure}

% \begin{figure}[htbp]
%     \centering
%     \includegraphics[width=0.6\linewidth]{figures_25_05_29/Figure_18_boundary.pdf}
%     \caption{Temporäres Bild für den Nachweis, dass wir Konvergenzrate 2 verlieren, wenn wir den Rand mit rein nehmen, längere Simulation kann gerne noch folgen.}
%     \label{fig:boundaryAdjustment_convergenceRate}
% \end{figure}

\begin{figure}[htbp]
    \centering
    \begin{subfigure}{0.48\textwidth}
        \centering
         \includegraphics[width=\linewidth]{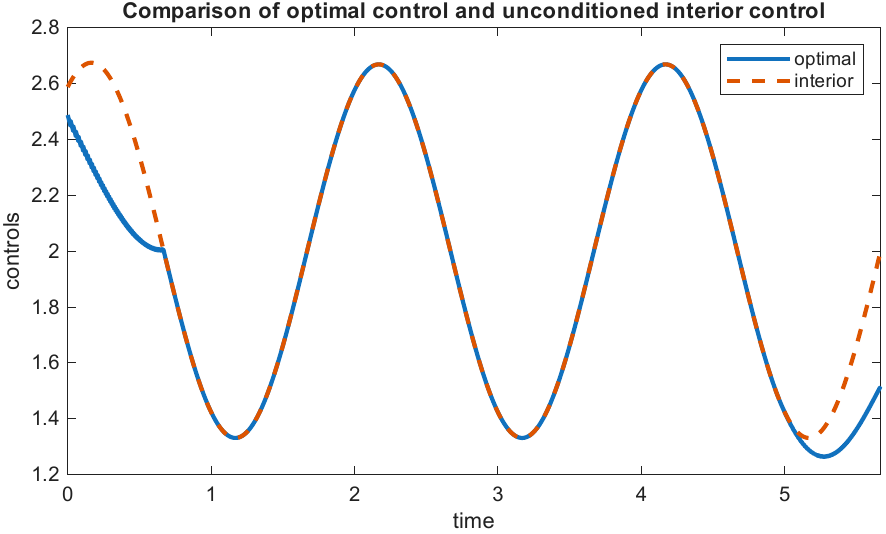}
    \end{subfigure}
    \hfill
    \begin{subfigure}{0.48\textwidth}
        \centering
         \includegraphics[width=0.97\linewidth]{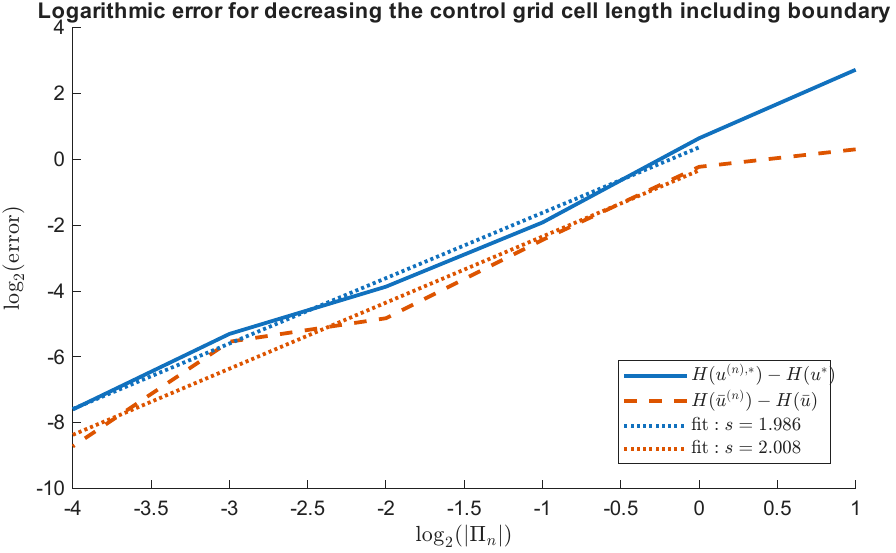}
    \end{subfigure}
    \caption{Comparison of the optimal control with the unconditioned interior expression (left) and numerical convergence rates including boundary effects (right).}
    \label{fig:boundaryAdjustment}
\end{figure}

\subsection{Beyond the linear theory: a nonlocal transport model}
\label{sec:numNonlocal}
We finally investigate whether the mean-velocity proxy remains useful when the explicit linear theory is no longer available. We consider the time-independent stochastic nonlocal velocity model motivated by \cite{Boehme2026},
\begin{align}
\rho_t
+
\partial_x\!\left(
\rho\,
(W_\eta *_x v_\varepsilon(\rho))
\right)
=
0,
\label{eq:nonlocalModel}
\end{align}
where
\[
v_\varepsilon(\rho)
=
\max\{0,v(\rho)+\varepsilon\},
\qquad
v(\rho)=1-\frac{\rho}{2},
\qquad
\varepsilon\sim\mathcal U([-\tau,\tau]),
\]
and
\[
W_\eta(x)
=
\frac{3}{2\eta}
\left(1-\frac{x^2}{\eta^2}\right),
\qquad
\eta=0.2,
\]
on the support of the kernel. The perturbation $\varepsilon$ is random but constant in time for each realization.

As a deterministic proxy, we replace the random velocity law by its pointwise expectation
\[
\bar v(\rho):=\mathbb E[v_\varepsilon(\rho)].
\]
We consider the spatial domain $[0,1]$ and set $T=10$. For the numerical discretization, we use the upwind-type schemes presented in \cite{Boehme2026} with step sizes $\Delta x=0.05$ and $\Delta t=\Delta x/2$. Both the stochastic and proxy control problems are solved numerically with \texttt{fmincon}. The objective is evaluated after an initial transient phase on $[3,T]$. Whenever an expectation is approximated by Monte Carlo simulation, we use 200 samples.

% \begin{figure}[ht!]
%     \centering
%     \includegraphics[width=0.5\linewidth]{figures_25_05_29/nonlocal_comp_inflow_outflow.pdf}
%     \label{fig:nonlocalF1}
% \end{figure}

\begin{figure}[htbp]
    \centering
    \begin{subfigure}{0.48\textwidth}
        \centering
    \includegraphics[width=0.96\linewidth]{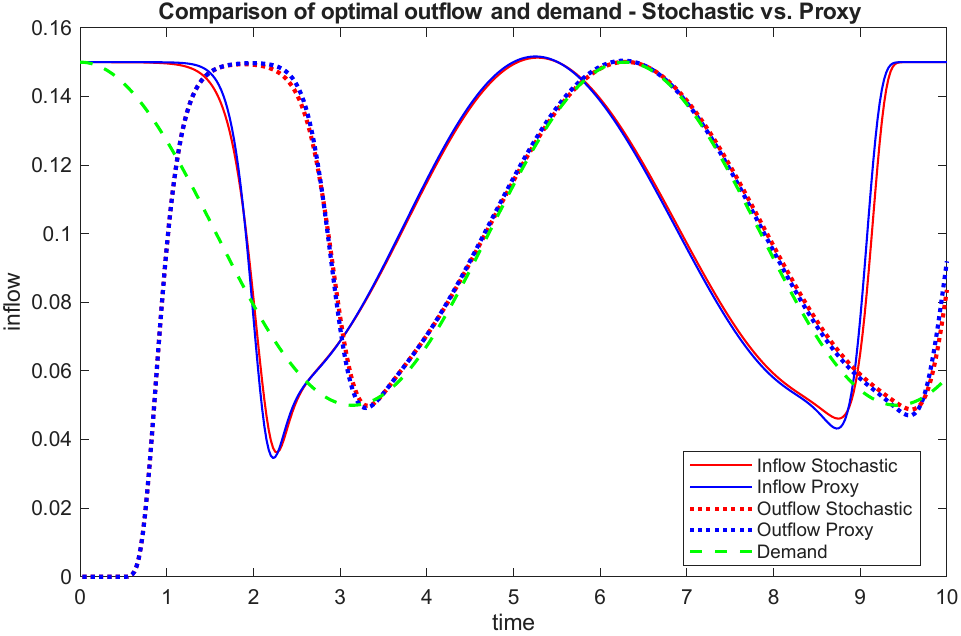}
    \end{subfigure}
    \hfill
    \begin{subfigure}{0.48\textwidth}
        \centering
         \includegraphics[width=\linewidth]{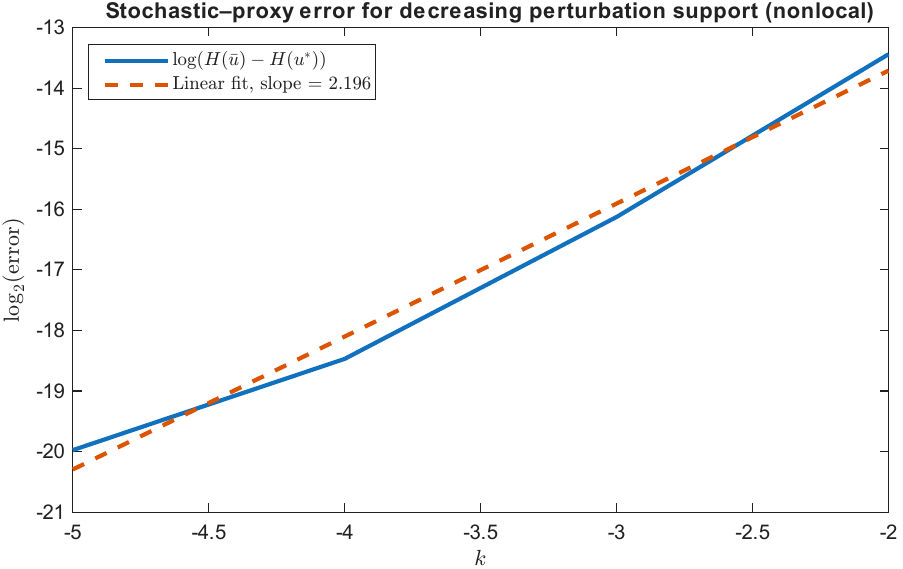}
    \end{subfigure}
    \caption{Comparison of the outflows generated by the stochastic and proxy controls with the mean demand (left), and the corresponding error and empirical convergence rate for the nonlocal model (right).}
    \label{fig:nonlocalF1}
\end{figure}

For the representative perturbation $\varepsilon\sim\mathcal U([-0.2,0.2])$, the reported objective values are
\[
4.062\times10^{-4}
\quad\text{for the stochastic control,}
\qquad
4.457\times10^{-4}
\quad\text{for the proxy control,}
\]
confirming, as in the local case, that the proxy introduces an additional error.

% \begin{figure}[ht!]
%     \centering
%     \includegraphics[width=0.5\linewidth]{figures_25_05_29/nonlocal_var_err.pdf}
%     \caption{Existing error comparison for decreasing support of the nonlocal velocity perturbation. No convergence order is claimed without additional theory or recomputed evidence.}
%     \label{fig:nonlocalF2}
% \end{figure}

The right panel of Figure~\ref{fig:nonlocalF1} examines the behavior of the stochastic and proxy controls as the support of the perturbation shrinks according to
\[
\varepsilon\sim\mathcal U([-2^{k},2^{k}]),
\qquad k=-6,\ldots,-2.
\]
The purpose of this experiment is to test whether the stochastic and proxy solutions approach each other as the uncertain velocity law concentrates. Since no analogue of Lemma~\ref{lem:proxyError} has been proved for the nonlinear nonlocal model, no theoretical convergence order is asserted. The numerical results exhibit an empirical rate of approximately 2.2, suggesting that behavior similar to the local second-order convergence may also occur in the nonlocal setting. Establishing such a rate theoretically, however, remains an open question.

The main practical motivation for using the proxy in the absence of an explicit optimal control formula, as available in the local case, is its computational efficiency. For the stochastic model, each evaluation of the sample-average objective requires one forward simulation for every Monte Carlo realization, making the stochastic optimization substantially more expensive. In contrast, each objective evaluation for the proxy model requires only a single deterministic forward solve. On a machine equipped with six CPU cores, the stochastic optimization based on 200 Monte Carlo samples required approximately 60 times the computational time of the proxy optimization. For larger numbers of Monte Carlo samples, this computational advantage of the proxy is expected to become even more relevant. Thus, in settings without explicit optimal controls and for small perturbation variances, the proxy provides a computationally efficient approximation that remains close to the optimal stochastic solution.

\section{Conclusion}
\label{sec:conclusion}

We have studied optimal inflow control for a linear transport equation under simultaneous uncertainty in downstream demand and transport velocity. The random velocity induces a random travel time and thereby changes the relation between an inflow decision and the demand observed at its arrival time. Retaining a fixed downstream observation window additionally gives rise to temporal boundary effects. We derived explicit optimal controls in continuous time and for piecewise constant control spaces and characterized their relation under temporal refinement.

Our error analysis separates the effects of demand uncertainty, random travel times, and control discretization. The contribution generated by the uncertain transport velocity admits a bound that is linear in its variance, whereas the additional performance loss of the deterministic mean-velocity proxy admits a higher-order, quadratic bound. Temporal discretization likewise yields a quadratic error bound under suitable regularity assumptions. Moreover, the optimal control is Lipschitz stable with respect to perturbations of the velocity distribution in the Wasserstein distance. These results provide a theoretical explanation for why the computationally inexpensive mean-velocity proxy can remain accurate even when transport dynamics are uncertain. The numerical experiments illustrate the different error mechanisms and explore the proxy strategy for a nonlinear nonlocal transport model beyond the explicit linear setting.

Several extensions arise naturally. One direction is to consider time-dependent stochastic transport velocities, for which travel times become path dependent and the explicit characterization developed here is no longer available. A second direction concerns transport networks, where uncertain travel times interact with routing and junction conditions. Further questions include extensions to nonlinear and nonlocal conservation laws and the treatment of correlations between demand and transport uncertainty.

%\begin{quote}
%\textcolor{blue}{%
%\textbf{Declaration of generative AI and AI-assisted technologies in the manuscript preparation process.}
%During the preparation and revision of this manuscript, the authors used ChatGPT (OpenAI) to assist with language editing, manuscript organization, literature exploration, and the discussion and refinement of mathematical arguments. All mathematical statements, proofs, references, and numerical results were independently reviewed and verified by the authors. The authors take full responsibility for the content of the manuscript.}
%\end{quote}

\section*{Declaration of generative AI and AI-assisted technologies}
During the preparation and revision of this manuscript, the authors used ChatGPT, GPT-5.6 Sol (OpenAI), for language editing, improvements in clarity and organization, and exploratory discussion during manuscript development. All mathematical arguments, results, proofs, references, and numerical results were independently verified by the authors. The authors take full responsibility for the content of the manuscript.

\section*{Data Availability Statement}
%The code used for the numerical experiments is available from the authors upon reasonable request.
%
The research codes associated with this article are available in \url{https://github.com/Tommek28/Optimal-Inflow-Control-for-Uncertain-Velocities}, under the reference \cite{Schillinger2026Code}.

\printbibliography

\end{document}